\documentclass[12pt,reqno]{amsart}
\usepackage{amsmath,amssymb,amsthm,mathtools,wasysym,calc,verbatim,enumitem,tikz,pgfplots,url,hyperref,mathrsfs,cite,fullpage,bbm}

\usepackage{comment}

\newtheorem{theorem}{Theorem}

\newtheorem{lemma}[theorem]{Lemma}

\newtheorem{fact}[theorem]{Fact}

\newtheorem*{claim*}{Claim}

\theoremstyle{remark}
\newtheorem*{remark*}{Remark}

\numberwithin{theorem}{section}

\renewcommand{\phi}{\varphi}

\renewcommand{\leq}{\le}
\renewcommand{\geq}{\ge}
\newcommand{\one}{\mathbf{1}}

\newcommand{\eps}{\varepsilon}

\newcommand{\cF}{\mathcal F}

\newcommand{\cP}{\mathcal P}

\newcommand{\E}{\mathbb{E}}

\def\1{\mathbbm{1}}

\def\g{{\gamma}}

\renewcommand{\le}{\leqslant}
\renewcommand{\ge}{\geqslant}
\renewcommand{\P}{\mathbb{P}}

\newcommand{\R}{\mathbb R}

\newcommand{\bX}{\mathbf X}

\newcommand{\Vol}{\mathrm{Vol}}
\newcommand{\Gbar}{\bar{G}}

\newcommand{\dist}{\operatorname{dist}}

\title
{A new lower bound for sphere packing} 
\author{Marcelo Campos, Matthew Jenssen, Marcus Michelen, Julian Sahasrabudhe}

\address{University of Cambridge. Trinity College.}
\email{mc2482@cam.ac.uk}

\address{King's College London, Department of Mathematics.}
\email{matthew.jenssen@kcl.ac.uk}
\address{University of Illinois, Chicago. Dept of Mathematics, Statistics and Computer science.}
\email{michelen@uic.edu}
\address{University of Cambridge. Department of Pure Mathematics and Mathematical Statistics.} 
\email{jdrs2@cam.ac.uk}

\begin{document}
	\begin{abstract}
 We show there exists a packing of identical spheres in $\mathbb{R}^d$ with density at least
 \[
  (1-o(1))\frac{d \log d}{2^{d+1}}\, ,
 \]
 as $d\to\infty$.
This improves upon previous bounds for general $d$ by a factor of order $\log d$ and is the first asymptotically growing improvement to Rogers' bound from 1947.
	\end{abstract}	
\maketitle

\vspace{-2em}

\section{Introduction}

In this paper we consider the classical \emph{sphere packing problem}: what is the maximum proportion of $\R^d$ that can be covered by a collection of disjoint spheres of volume $1$? Determining the value of this maximum, denoted $\theta(d)$, has proven to be an exceedingly challenging problem
and is only known in dimensions $d\in \{1,2,3,8,24\}$. While the $d=1$ case is trivial and $d=2$ is classical \cite{thue1911dichteste}, the remaining cases are the result of a series of extraordinary breakthroughs: dimension 3 was a landmark achievement of Hales \cite{hales2005proof}, resolving the Kepler conjecture from 1611. Dimensions 8 and 24 were resolved only recently due to the major breakthroughs of Viazovska~\cite{viazovska2017sphere}, in dimension $8$, and then Cohn, Kumar, Miller, Radchenko, and Viazovska~\cite{cohn2016sphere} in dimension $24$. (See~\cite{cohn2016conceptual} for a beautiful exposition of these developments).

Here we study $\theta(d)$ for \emph{large} $d$ where the nature of optimal sphere packings remains almost entirely mysterious, with many basic questions left unresolved. Here the trivial lower bound is $\theta(d) \geq 2^{-d}$, which follows by considering any saturated collection of disjoint spheres. In 1905, this was improved by a factor of $2+o(1)$ by Minkowski~\cite{Min05}. In 1947, Rogers~\cite{rogers1947existence} made the first asymptotically growing improvement to the trivial lower bound showing that $\theta(d)\geq (1+o(1))cd2^{-d}$ where $c = 2/e$. Since then, a number of improvements have been made to the constant factor $c$. Davenport and Rogers~\cite{davenport1947hlawka} showed $c = 1.68$; Ball~\cite{ball1992lower}, some 45 years later, improved the bound to $c = 2$; and Vance~\cite{vance2011improved} showed $c = 6/e$ when $d$ is divisible by $4$. The most recent improvement was made by Venkatesh~\cite{venkatesh2012note} who showed that one can take $c= 65963$. Moreover he showed that one can obtain an additional $\log \log d$ factor along a sparse sequence of dimensions.  On surveying this sequence of results Venkatesh writes: `\emph{It is striking that the quoted papers use quite different methods and yet all arrive at the same linear improvement
on Minkowski’s bound.}' In this paper, we go beyond this barrier and improve Minkowski's bound by a factor of order $d \log d$ in all dimensions. 

\begin{theorem}\label{thm:main}
   \[
    \theta(d) \geq (1-o(1)) \frac{d\log d}{2^{d+1}}\, ,\] as $d\to\infty$. 
\end{theorem}
In each of the aforementioned papers, the sphere packings arise from collections of spheres that are centred on \emph{lattices}. By contrast, our packings are highly \emph{disordered} and constructed by recursively adding random spheres to $\R^d$, in rounds, that are centred on the points of an appropriately modified Poisson process. Until now, nothing has been shown to beat (highly symmetric) lattice constructions in high dimensions, even though some have speculated that \emph{optimal} packings may in fact be ``disordered.'' Indeed, Cohn~\cite{cohn2016packing} observes, `\emph{One philosophical quandary is that too much structure seems to make high dimensional packing bad, but the known lower bounds all rely on some sort of heavy structure... The trade off is that structure hurts density but helps in analyzing packing}'.

We point out that, while our bounds improve the situation for $\theta(d)$,
there still remains an exponential gap between the upper and lower bounds; the best upper bound in high dimensions is $\theta(d) \le 2^{- (.599\dots +o_d(1)) d}$ which is due to the 1978 work of Kabatjanski\u\i\, and Leven\v ste\u\i n~\cite{kabatiansky1978bounds} and has only been improved by a multiplicative constant factor in the years since by Cohn and Zhao~\cite{cohn2014sphere} and Sardari and Zargar~\cite{sardari2023new}.

\subsection{Spherical codes} A further consequence of our method is that it adapts naturally to the problem of constructing large \emph{spherical codes} in high dimension, which is not true of the lattice constructions discussed above. A spherical code is a collection of points on the unit sphere $S^{d-1}$ of pairwise angle at least $\theta$. That is, a packing of spherical caps of angular radius $\theta/2$. We let $A(d,\theta)$ be the largest possible size of such a collection and  $s_d(\theta)$ be the surface area of a spherical cap of angular radius $\theta$ normalised by the surface area of $S^{d-1}$. 

While the case $\theta \ge \pi/2$ was settled exactly in 1955 by Rankin~\cite{rankin1955closest}, the state of knowledge for $\theta < \pi/2$ is similar to the case of sphere packings, with exponential gaps between the upper and lower bounds. A lower bound of $ A(d,\theta)  \geq 1/s_d(\theta)$ comes from simply considering a saturated packing of spherical caps, while the best known upper bounds come from the important paper of Kabatjanski\u\i\, and Leven\v ste\u\i n~\cite{kabatiansky1978bounds}.

The lower bound was recently improved by the second author, Joos and Perkins~\cite{jenssen2018kissing} who showed that $A(d,\theta) \geq c d /s_d(\theta)$ for a constant $c>0$, the value of which was subsequently improved by Fern\'{a}ndez, Kim, Liu and  Pikhurko~\cite{fernandez2021new}.  The methods of this paper can be used to give a further asymptotically growing improvement to the lower bound. 

\begin{theorem}\label{thm:sphere-codes} If $\theta \in (0,\pi/2),$ then
\[ A(d,\theta)\geq (1-o(1)) \frac{d\log d}{2 s_d(\theta)} \]
 as $d \rightarrow \infty$.
\end{theorem}

We remark that the case of $\theta=\pi/3$ is of particular interest as it is equivalent to the well-known \emph{kissing number problem}, which is to determine the maximum number $K(d)$ of non-overlapping unit spheres that can touch a single unit sphere. Theorem~\ref{thm:sphere-codes} applied with $\theta=\pi/3$ implies that \[K(d)\geq (1-o(1))\sqrt{\frac{\pi}{8}}\left(\frac{2}{\sqrt{3}}\right)^{d-1} d^{3/2}\log d \, .\]

The proof of Theorem~\ref{thm:sphere-codes} is a simple adaptation of that of Theorem~\ref{thm:main} and so we defer the details to the appendix.

\subsection{Amorphous sphere packings in physics}

In physics, sphere packings serve as a model to investigate the phases of matter, where random sphere packings at a given density model the structure of matter at a given temperature. 
In dimension $3$, for instance, it is believed that random sphere packings transition from appearing ``gas-like'' at low density to ``lattice-like'' at high density, paralleling the phase transition between states of matter. However, rigorously demonstrating that this phase transition occurs remains a major open problem in the field (see \cite{lowen2000fun} and the references therein). 

Physicists have also devoted enormous effort to analysing sphere packings in high dimensions, with the aim of providing a more tractable analysis than in low dimensions, and in order to use the powerful machinery of equilibrium statistical physics to generate rich predictions~\cite{frisch1999high, frisch1985classical, frisch1987nonuniform, parisi2010mean, parisi2006amorphous, charbonneau2021three, torquato2006new, torquato2008}. 
Here, the important qualitative distinction is between sphere packings that are \emph{crystalline}, meaning that  they exhibit long-range ``correlations'', and \emph{amorphous}, meaning they don't have any such correlations.  For example, lattice packings are extreme instances of crystalline packings where the entire structure is determined by a basis.

In their seminal work on applying the replica method to the structure of high-dimensional sphere packings, Parisi and Zamponi~\cite{parisi2010mean,parisi2006amorphous} predicted that the largest density of amorphous packings in $d$ dimensions is $(1+o(1))(d\log d) 2^{-d}$, that is, a factor of $2$ larger than our lower bound from Theorem~\ref{thm:main}. While there is no agreed-upon rigorous definition of ``amorphous,'' it seems likely that any such definition would be satisfied by our construction for Theorem~\ref{thm:main}, which enjoys extremely fast decay of correlations.

\subsection{Methods and related work} 
Our proof can be described in broad steps as follows. First we discretize space via a Poisson point process at a carefully chosen intensity, followed by an alteration step where we impose additional uniformity properties on the discrete point set. Call the output of this process $X\subseteq \R^d$. We consider the graph $G(X,r)$  whose vertex set is $X$ and two distinct $x,y \in X$ are joined by an edge whenever $\|x-y\|\leq 2r$. An independent set in $G(X,r)$ is a packing of spheres of radius $r$ and so our goal becomes to find a large independent set in this graph. For this step we develop the following new graph theoretic tool that we believe is of independent interest.

 For a graph $G$ we let $\alpha(G)$ denote the size of the largest independent set in $G$. We let $\Delta(G)$ denote the maximum degree of $G$ and let $\Delta_2(G)$ denote the maximum codegree, that is, the maximum number of common  neighbours a pair of distinct vertices in $G$ can have. 
 

\begin{theorem}\label{thm:ind-graph}
Let $G$ be a graph on $n$ vertices, such that $\Delta(G)\leq \Delta$ and $\Delta_2(G)\leq C\Delta (\log \Delta)^{-c}$. Then 
\[\alpha(G)\geq (1-o(1))\frac{n\log \Delta }{\Delta }\,,\]
where $o(1)$ tends to $0$ as $\Delta\to \infty$ and we can take $C = 2^{-7}$ and $c = 7$.
\end{theorem}

This theorem parallels a classical result of Shearer \cite{shearer1983note} who, building on an earlier breakthrough of Ajtai Koml\'os and Szemer\'edi \cite{ajtai1980note}, showed the same conclusion holds when $G$ is a \emph{triangle-free} graph of maximum degree $\Delta$. 
We remark that one can apply Shearer's theorem as a black box (along with a standard technique) to find an independent set of density $(\log \log \Delta)/\Delta$ in the context of Theorem~\ref{thm:ind-graph}. However, this bound is far too weak for our application. Heuristically speaking, the  discretization of space we construct is \emph{too dense} for an application of these classical results. Our new tool (Theorem~\ref{thm:ind-graph}) allows us to consider discretizations with $d^{\Omega(d)}$ points per unit volume, whereas an application of Shearer's theorem is optimal for discretizations with $e^{O(d)}$ points per unit.

We also note that one can see that Theorem~\ref{thm:ind-graph} is sharp, up to the constants $C,c$, by taking $\eta \approx (\log \Delta)^{-1}$ and considering a graph on $n$ vertices which is formed by taking the union of $n/(\eta \Delta)$ disjoint cliques of size $\eta \Delta $ and then adding (on top) a random regular graph of degree $(1-\eta)\Delta$. This graph has codegrees bounded by $\eta\Delta$ and $\alpha(G) \leq n/(\eta \Delta)$. 

While Theorem~\ref{thm:ind-graph} is novel, this is not the first instance of a Shearer-like theorem being used in the context of sphere packing. Krivelevich, Litsyn, and Vardy~\cite{krivelevich2004lower} used Shearer's theorem to give a packing of density $\Omega(d2^{-d})$, and the second author, Joos and Perkins~\cite{jenssen2019hard}, inspired by methods in the graph theoretic setting \cite{davies2015independent, davies2018average}, proved that a random sample from the ``hard-sphere model'' (at an appropriate temperature) gives packings of density $\Omega( d2^{-d})$ in expectation.

Interestingly, it is reasonable to suspect that Theorem~\ref{thm:ind-graph} can be improved by a factor of $2$. This would result in a factor $2$ improvement in our lower bound, which would match the prediction of Parisi and Zamponi mentioned above. This missing factor also connects with a major open problem in combinatorics, where the corresponding factor of $2$ is believed to hold in the setting of Shearer's theorem, when $G$ is triangle free. In this latter case, this ``missing'' factor is believed to be one of the two missing factors of $2$ that separate the upper and lower bounds on the off-diagonal Ramsey numbers $R(3,k)$, see \cite{BK-triangle-free,FGM-triangle-free}.

To prove Theorem~\ref{thm:ind-graph} we iteratively build up an independent set by sampling a random set of vertices at density $\g/\Delta$, for $\g \ll 1$, and then removing it and all of its neighbours from the graph. This technique has come to be known as the \emph{R\"{o}dl nibble} after the important work of R\"{o}dl \cite{rodl}, and it in fact goes back even further to the work of Ajtai, Koml\'os and Szemer\'edi \cite{ajtai1981dense}. The extra twist in our approach is that we additionally \emph{add} edges after each nibble step, to ensure that our graph remains approximately regular.

\subsection{Notation}
 We let $\|\cdot\|$ denote the Euclidean $2$-norm. For $x\in \R^d$ and $r>0$, we let $B_x(r)=\{y\in \R^d: \|x-y\|\leq r\}$ denote the closed ball centred at $x$ with radius $r$. For a measurable $S\subseteq \R^d$, we write $\Vol(S)$ for the usual Lebesgue measure of $S$. We write $\bX\sim \mathrm{Po}_\lambda(S)$ to denote that $\bX$ is a Poisson point process of intensity $\lambda$ on $S$. Given a set $V$ and $0\leq p\leq 1$, a $p$-random set $A$ is a random set sampled by including each element $v\in V$ independently with probability $p$.

Given a finite graph $G=(V,E)$, and $v\in V$, we let $N_G(v)=\{u\in V : \{u,v\}\in E\}$ denote the neighbourhood of $v$ in $G$. For a set $U\subseteq V$, we let $N_G(U)=\bigcup_{u\in U}N_G(u)$. For $u,v\in V$, we let $d_G(v)=|N_G(v)|$ denote the degree of $v$ and let $d_G(u,v)=|N_G(u)\cap N_G(v)|$ denote the codegree of $u$ and $v$. We write $\Delta(G)=\max_{v\in V}d_G(v)$ and $\Delta_2(G)=\max_{u,v\in V, u\neq v}d_{G}(u,v)$ for the maximum degree and codegree of $G$ respectively. Given $X,Y\subseteq V$ disjoint, let $E_G(X,Y)=\{\{x,y\}: x\in X, y\in Y \}$ and let $e_G(X,Y)=|E_G(X,Y)|$. We will often drop $G$ from the subscript in this notation if the graph $G$ is clear from the context. We will often write $v\in G$ to mean $v\in V$ and write $|G|$ to mean $|V|$. For any $A\subset V$ we define $G\setminus A$ to be the graph with vertex set $V\setminus A$ and edge set $E\cap \binom{V\setminus A}{2}$. For $e\in \binom{V}{2}$, we define $G+e$ to be the graph with vertex set $V$ and edge set $E\cup \{e\}$. For $u,v\in V$, we write $\dist_{G}(u,v)$ for the usual graph distance between $u,v$, that is, the length of the shortest path from $u$ to $v$ in $G$. We let $\alpha(G)$ denote the size of the largest independent set in $G$. 

\section{Proof of Theorem \ref{thm:main}}
Let us begin with a precise formulation of the problem at hand. Throughout, we let $r_d$ denote the radius of the ball of volume $1$ in $\R^d$. We call a collection of points $\mathcal P\subseteq \R^d$ a sphere packing if $\|x-y\|_2\geq 2r_d$ for all distinct pairs $x,y\in \mathcal{P}$. The sphere packing density $\theta(d)$ is then defined by
\begin{align*}
\theta(d) &= \sup_{\mathcal P} \limsup_{R \to \infty}  \frac{ |\mathcal P \cap B_0(R)| }{\Vol (B_0(R))} \, ,
\end{align*}
where the supremum is over the set of all possible packings $\cP$. It is well-known that the above supremum and limit supremum can be interchanged, where now the supremum is over all packings $\cP\subseteq B_{0}(R)$ (see, e.g., \cite[Section 1]{cohn2016packing} or \cite{groemer1963existenzsatze}). This formulation will be more convenient for our purposes.

Throughout this section $\Omega\subseteq \R^d$ will be a bounded and measurable set (which can be thought of as a large ball). 
We will apply Theorem \ref{thm:ind-graph} to a finite collection of ``candidate'' spheres in $\Omega$.  For a finite $X \subseteq \R^d$ and $r >0$ we define the graph $G(X,r)$ to be the graph on vertex set $X$ where two distinct $x,y \in X$ are connected whenever $\|x-y\|\leq 2r$.
Thus an independent set in $G(X,r)$ is a packing of spheres of radius $r$.  

In the following lemma, we show that there exists a finite set $X\subseteq \Omega$ such that $X$ is reasonably large and `uniform' in the sense that $G(X,r_d)$ has small maximum degree and codegree (allowing for an application of Theorem~\ref{thm:ind-graph}). We construct $X$ by sampling a Poisson process and then removing points that violate our uniformity conditions.  In Section \ref{sec:preprocessing} we will prove the following:

\begin{lemma}\label{prop:preprocessing}
Let $\Omega\subseteq \R^d$ be bounded and measurable.
    For all $d \geq 1000$, there exists $X \subset \Omega$ so  that 
\[ |X| \geq   \big(1-1/d\big)  \frac{\Delta}{2^d}\Vol(\Omega), \quad \text{ where } \quad \Delta = \bigg(\frac{\sqrt{d}}{4\log d}\bigg)^d, \]
and if $G = G(X,r_d)$ we have
\begin{equation} 
\label{it:prop-codegree}
    \Delta(G) \leq \Delta\big(1 + \Delta^{-1/3}  \big) \qquad  \text{ and } \qquad \Delta_2(G) \leq \Delta  \cdot e^{-(\log d)^2 /8}\, . \end{equation}
\end{lemma}

Note that $e^{-(\log d)^2} \approx (\log \Delta )^{-\log\log \Delta}$; thus the codegree condition at \eqref{it:prop-codegree} is appropriate for the application of our Theorem~\ref{thm:ind-graph}, with room to spare. Therefore we are able to deduce Theorem~\ref{thm:main} from Theorem~\ref{thm:ind-graph} and Lemma~\ref{prop:preprocessing}.

\begin{proof}[Proof of Theorem \ref{thm:main}]
Set $\Omega = B_0(R)$, the ball of radius $R$ centered at the origin. 
    It is sufficient to prove that for $R>0$ we can place $(1 - o(1)) \Vol(\Omega)\frac{ d \log d}{2^{d+1}}$ points in $\Omega$ that are each at pairwise distance at least $2r_d$.  

    Apply Lemma \ref{prop:preprocessing} to obtain $X \subset \Omega$ 
with $|X| \geq (1-o(1))\Vol(\Omega) \Delta2^{-d}$ and for which
\[ \Delta(G) \leq (1+o(1))\Delta \quad \text{ and } \quad \Delta_2(G)\leq \Delta  \cdot e^{-(\log d)^2 /8}\leq \Delta (\log \Delta )^{-\omega(1)}\, ,\]
where $G = G(X,r_d)$.  For $R>0$ we now apply Theorem~\ref{thm:ind-graph} to find an independent set $I \subseteq X$ for which 
\[ |I| \geq  (1-o(1))\frac{|X|\log \Delta }{\Delta} \geq (1-o(1)) \Vol(\Omega)\frac{ d\log d}{2^{d+1}}\, .\]
Now note that $I$ corresponds to a set of points in $\Omega$ of pairwise distance at least $2r_d$, completing the proof. 
\end{proof}

\subsection{Completing the discretization step: Proof of Lemma \ref{prop:preprocessing}} \label{sec:preprocessing}

Recall that we fixed $\Omega\subseteq \R^d$ bounded and measurable. 
We will prove Lemma~\ref{prop:preprocessing} by sampling a Poisson point process $\bX \subset \Omega$  with intensity $\lambda = 2^{-d}\Delta=\left(\frac{\sqrt{d}}{8\log d}\right)^d$. We then remove points $x \in \bX$ which satisfy
\begin{equation}\label{eq:bad-conditions} |\bX \cap B_x(2r_d)| \geq \Delta\big(1 + \Delta^{-1/3}\big) \quad \text{ or } \quad |\bX \cap B_x(2r_d) \cap B_y(2r_d)| \geq \Delta \cdot e^{- (\log d)^2 / 8}\, ,\end{equation}
for some $y \in \bX$. These two ``bad'' events correspond, respectively, to the degree of a vertex being too large in $G(X,r_d)$ and to $x$ being in a pair $x,y$ for which the codegree is large. 
We note that removing a point from $\bX$ \emph{does not} create another instance of these bad events and so we never create new bad events by deleting points.   Throughout this subsection we assume $d \geq 1000$.

To prove that we don't delete too many points, we make repeated use of the following property of Poisson point processes, known as the Mecke equation (see \cite{last2017lectures}). For a bounded and measurable set $\Lambda\subseteq \R^d$ and events
$(A_x)_{x \in \Lambda}$ we have 
\begin{equation}\label{eq:palm} \E \big| \big\{ x \in \bX \cap \Lambda : A_x \text{ holds for }\bX \big\} \big| = \lambda  \int_{\Lambda}  \P\big(A_x \text{ holds for }\bX \cup \{x\}\big)\, dx. \end{equation}
We will also use a basic bound on the tail of a Poisson random variable $Y$ (see e.g.\ \cite[Theorem 2.1]{janson2011random}): for all $t > 0$ we have \begin{equation} \label{eq:Poisson-tail}
    \P\big( Y - \E Y \geq t \cdot \E\, Y \big) \leq \exp\big(-\min\{ t, t^2\}  \cdot \E\, Y / 3\big) \,.
\end{equation}

The following lemma shows that only a small fraction of points in $\bX$ satisfy the first bad condition in \eqref{eq:bad-conditions}. The proof uses the concentration of a Poisson point process to say it approximates volumes well.

\begin{lemma}\label{lem:prec-bad-1} Let $\bX \sim \mathrm{Po}_\lambda\big(\Omega\big)$.  We have  
    $$\E\left|\big\{x \in \bX :  |\bX \cap B_x(2r_d)| \geq \Delta\big(1 + \Delta^{-1/3}\big) \big\}   \right| \leq (2d)^{-1} \cdot \E|\bX| \, .  $$
\end{lemma}
\begin{proof}
    Apply \eqref{eq:palm} to note that for each $s > 0$ we have  
    \begin{equation}\label{eq:palm1}
        \E\left|\left\{x \in \bX : |\bX \cap B_x(2r_d)|\geq s\right\} \right| = \lambda \int_\Omega \P\big(|\bX \cap B_x(2r_d)|\geq s - 1 \big)\,dx\, .
   \end{equation} 
   Now note for fixed $x \in \Omega$ we have that $|\bX \cap B_x(2r_d)|$ is a Poisson random variable of mean at most $\lambda 2^d  = \Delta$. Then apply \eqref{eq:Poisson-tail} with $t=\Delta^{-1/3}-\Delta^{-1}$ to see
    \begin{equation*} \P\left(|\bX \cap B_x(2r_d)| \geq \Delta+\Delta^{2/3} - 1\right) \leq \exp\left(- \Delta^{1/3}  / 4 \right)  \leq (2d)^{-1}
    \end{equation*} and use this in \eqref{eq:palm1} to complete the proof.
\end{proof}

To control the second bad event we need the following basic estimate
\begin{align}\label{eq:BasicEst}
\left(\frac{\pi e t^2}{d} \right)^{d/2}\leq \Vol(B_0(t)) = \frac{\pi^{d/2}}{\Gamma(\frac{d}{2} + 1)}t^d \leq \left(\frac{2\pi e t^2}{d} \right)^{d/2} 
\end{align}
where the bounds are via Stirling's formula for $d \geq 4$.  From the lower bound and the inequality $\pi e>8$ we also see that $r_d \leq \sqrt{d/8}$. We also need the following basic geometric fact to bound the volume of the interesection of two balls, which corresponds to the codegree. 

\begin{fact}\label{fact:sphere-intersection} For $t \geq 0$, let $x,y \in \R^d$ satisfy $\|x - y\| \geq t$. Then $$\Vol(B_x(2r_d) \cap B_y(2r_d)) \leq 2^d e^{-t^2/4}\,.$$
\end{fact}
\begin{proof}
We observe that  
\begin{equation}\label{eq:Ball-inequal} B_x(2r_d) \cap B_y(2r_d) \subset B_{(x+y)/2}\big( (4r_d^2 - t^2/4)^{1/2}\big )\,,\end{equation}
since for all $x,y,z$ we have 
\[ 4\|z - (x+y)/2\|^2 = 2\|z - x\|^2 + 2\|z - y\|^2 - \|x - y\|^2 \,,  \] by expanding the norms as inner products.
Thus, using \eqref{eq:Ball-inequal}, gives
$$\Vol(B_x(2r_d) \cap B_y(2r_d)) \leq \Vol\left(B_{0}\left(\sqrt{4r_d^2 - t^2/4}\right)\right) = 2^d \left(1 - \frac{t^2}{16 r_d^2}\right)^{d/2} \leq 2^d\exp\left(-  \frac{t^2 d}{32 r_d^2}\right)\,.$$
Using the fact that $r_d \leq \sqrt{d/8}$ completes the proof.
\end{proof}
 
We are now ready to bound the expected number of vertices which have a bad codegree in our graph. The proof is similar to the proof of Lemma~\ref{lem:prec-bad-1}, using the fact that the Poisson point process approximates volumes well.
\begin{lemma}\label{lem:prec-bad-2} 
    Let $\bX \sim \mathrm{Po}_\lambda\big( \Omega \big)$ and put $\eta = e^{-(\log d)^2/8}$. We have  
    $$\E\, \big|\big\{x \in \bX : |\bX \cap B_x(2r_d) \cap B_y(2r_d)| \geq  \eta \Delta  \text{ for some }y\in \bX\big\}   \big| \leq (2d)^{-1} \cdot \E |\bX | \,.  $$
\end{lemma}
\begin{proof}
Define $I_{x,y} = |\bX \cap B_x(2r_d) \cap B_y(2r_d)|$ and apply \eqref{eq:palm}, as we did in Lemma~\ref{lem:prec-bad-1}, to see it is enough to show for every $x\in \Omega$
    \[ \P\big(\exists \hspace{.5mm} y \in \bX : I_{x,y} \geq \eta \Delta -1\big) \leq (2d)^{-1}\,. \] For this, we apply Markov's inequality twice to bound 
 \begin{equation}\label{eq:bad-2-split}\P\big(\exists \hspace{.5mm} y \in \bX : I_{x,y} \geq\eta \Delta -1\big) \leq  \E\, |B_x(\log d) \cap \bX|  + \E\, \big|\big\{ y \in \bX \setminus B_x(\log d) :I_{x,y} \geq \eta \Delta-1 \big\}\big| \end{equation}
    and we deal with these two terms on the right-hand-side separately. The first term is
\[    \lambda \cdot \Vol(B_x(\log d)) 
        \leq \bigg(\frac{\sqrt{d}}{8\log d}\bigg)^d  \bigg(\frac{2 \pi e (\log d)^2}{d}\bigg)^{d/2} = \bigg(\frac{\sqrt{2\pi e}}{8} \bigg)^d \leq  (4d)^{-1} \, ,\]
 where for the first inequality we used~\eqref{eq:BasicEst} and recalled that $\lambda =\left(\frac{\sqrt{d}}{8\log d}\right)^d$.
For the second term in 
 \eqref{eq:bad-2-split}, we proceed as in the proof of Lemma~\ref{lem:prec-bad-1}. First, note we only need to consider $y \in B_x(4r_d)$, otherwise $I_{x,y} = 0$, trivially. We use the Mecke equation \eqref{eq:palm} (again) to express the second term in \eqref{eq:bad-2-split} as
 \[ \lambda \cdot \int \P\big( I_{x,y} \geq \eta \Delta-2 \big)\, dy   \leq   ( \E\,|\bX \cap B_x(4r_d)|\, ) \cdot \max_{y}\, \P\big(  I_{x,y} \geq \eta\Delta-2 \big),  \] 
 where both the integral and maximum are over $y\in B_x(4r_d)  \setminus B_x(\log d)$. Now note that 
 \[ \E\, I_{x,y} \leq \lambda \Vol(B_x(2r_d) \cap B_y(2r_d)) \leq \lambda \cdot 2^d e^{-(\log d)^2/4} \leq   \Delta \cdot e^{-(\log d)^2/4} ,\]   by Fact~\ref{fact:sphere-intersection}. Since $I_{x,y}$ is Poisson, we apply \eqref{eq:Poisson-tail} to see $$( \E\, |\bX \cap B_x(4r_d) |  )\cdot \P\left( I_{x,y} \geq \eta \Delta - 2\right) \leq 2^d \Delta \cdot\exp\left(- \Delta e^{-(\log d)^2 /16 } \right) \leq (4d)^{-1},$$ as desired.
\end{proof}

Now we may combine Lemma~\ref{lem:prec-bad-1} and Lemma~\ref{lem:prec-bad-2} to prove Lemma~\ref{prop:preprocessing}.
\begin{proof}[Proof of Lemma~\ref{prop:preprocessing}]
    Let $\bX$ be the Poisson point process in $\Omega$ of intensity $\lambda$. Let $S_1 \subset \bX$ and $S_2 \subset \bX$, be the points that satisfy the first and second properties in \eqref{eq:bad-conditions}, respectively. We set $X =  \bX \setminus (S_1 \cup S_2)$ and apply Lemma~\ref{lem:prec-bad-1} and Lemma~\ref{lem:prec-bad-2} to see  \[ \E|X| \geq \E |\bX| - \E |S_1| - \E |S_2| \geq (1-1/d) \cdot \E |\bX |. \]
 Noting that $\E|\bX| = \Vol(\Omega)(\sqrt{d}/(8\log d))^d = \Vol(\Omega) \Delta2^{-d}$, finishes the proof.
\end{proof}

\section{Controlling degrees and codegrees}
We now turn our attention toward the proof of Theorem~\ref{thm:ind-graph}. As noted in the introduction our strategy will be to iteratively build an independent set by sampling a random set of vertices at density $\g/\Delta$ and then removing it, and all of its neighbours from the graph. We refer to one step of this iteration as a ``nibble''.

The main goal of this section is to show that the degrees and codegrees in our graph decrease appropriately in each nibble step thereby setting us up for further steps in our iteration.   Throughout this section, we will assume that $G$ is a graph satisfying \begin{equation} \label{eq:G-assumptions}
    d_G(v) \in \{  \Delta-1 , \Delta\} \hspace{3mm} \text{ for all } v \in V(G) \hspace{3mm} \text{ and } \hspace{5mm}  \Delta_2(G) \leq \eta \Delta
\end{equation}
and our parameters $\gamma,\eta,\Delta$ satisfy 
\begin{equation}\label{eq:param-assumptions}
    \Delta \geq 1, \quad \gamma \leq 1/2, \quad \Delta^{-1/2}\leq \eta \leq \gamma^2/8\,.
\end{equation}

\begin{lemma}\label{lem:degs}
Let $\Delta,\gamma, \eta$ satisfy \eqref{eq:param-assumptions}, let $\alpha \in [2\gamma^2,\gamma]$ and let $G$ satisfy \eqref{eq:G-assumptions}.  
Let $A \subseteq V(G)$ be a $p$-random set, where $p = \g/\Delta$, and let $G'=G\setminus (A\cup N_G(A))$.  
Then for all vertices $v \in G$ we have
\begin{equation}\label{eq:concentration-degs}
\P\big(d_{G'}(v)\geq\left(1-\gamma+\alpha\right)d_G(v)\, \big\vert \, v \in G'\big)\leq  \exp\left(-\frac{\alpha^2}{32\gamma \eta}\right) ,
\end{equation} and for all distinct $u,v \in G$, we have 
\begin{equation}\label{eq:concentration-codegs}
\P\big( d_{G'}(u,v)\geq\left(1-\gamma+\alpha\right)\eta \Delta\, \big\vert\, u,v \in G' \big)\leq  \exp\left(-\frac{\alpha^2}{32\gamma \eta}\right)\, .
\end{equation}
\end{lemma}

To prove \eqref{eq:concentration-degs} in Lemma~\ref{lem:degs} we first calculate the expected value $\E( d_{G'}(v)\, \vert\, v \in G')$ and then prove concentration by a standard martingale idea: we expose vertices of $A \cap N(N(v))$,
one-by-one, and track how $|N(v) \setminus A|$ decreases. The key here is that each exposure can only remove so much from $N(v)$ by the codegree condition. This is formally captured by bounding the increments in the ``exposure'' martingale and using an appropriate martingale concentration inequality. The proof of \eqref{eq:concentration-codegs} will follow from a very similar argument.

\subsection{Expected degrees and codegrees}  In this section we calculate the expected degrees and codegrees of the surviving vertices. 

\begin{lemma}\label{lem:E-deg}  Let $\Delta,\gamma,\eta$ satisfy \eqref{eq:param-assumptions} and let $G$ be a graph satisfying \eqref{eq:G-assumptions}.  
Let $A \subseteq V(G)$ be $p$-random, where $p = \g/\Delta$, and let $G'=G\setminus (A\cup N_G(A))$.  Then, for all $u, v \in G$, we have
\[ \E\big( d_{G'}(v) \hspace{.4mm} \vert\, v \in G' \big) \leq \left(1-\gamma+\gamma^2\right)d_G(v)\, , \]
and 
\[\E \big( d_{G'}(u,v)\,|\,u,v\in G' \big) \leq \left(1 - \gamma + \gamma^2\right) d_G(u,v) \, . \] 
\end{lemma}
\begin{proof}
Conditioned on the event $v\in G'$, we see that $A$ is a $p$-random subset of the set $V(G)\backslash (N_G(v)\cup \{v\})$. Let $w\in N_G(v)$ and let $d_w=|N_G(w)\setminus (N_G(v)\cup \{v\})|$.
 From \eqref{eq:G-assumptions} we have  $d_G(w,v)\leq \eta \Delta$ and $d_G(w) \geq \Delta - 1$ and so 
 \[ d_w \geq d_G(w)-\eta\Delta-1 \geq (1-\gamma^2/2)\Delta\, , \]
 since $\Delta^{-1/2}\leq \eta \leq \gamma^2/8$.
 It follows that 
\[
\P\big( w \in G' \,|\, v \in G' \big)=\left(1-\frac{\gamma}{\Delta}\right)^{d_w}\leq 1- \frac{\gamma}{\Delta}d_w +\frac{\gamma^2 }{2\Delta^2}d_w^2\leq 1-\gamma  + \gamma^2\, ,
\] where for the last inequality we used
$(1-\gamma^2/2)\Delta \leq d_w\leq \Delta$. So summing over $w\in N_G(v)$ we get that $$\E\big( d_{G'}(v) \hspace{.4mm} \vert\, v \in G' \big) \leq \left(1-\gamma+\gamma^2\right)d_G(v)\, .$$

Similarly, conditioned on the event $u,v\in G'$, we see $A$ is a $p$-random subset of $V(G)\backslash (N_G(u)\cup N_G(v)\cup \{u,v\})$. Let $w\in N_G(u) \cap N_G(v)$ and let $d_w'=|N_G(w)\setminus (N_G(u)\cup N_G(v)\cup \{u,v\})|$.
 From \eqref{eq:G-assumptions} we have  $d_G(w,u)\leq \eta \Delta$, $d_G(w,v)\leq \eta \Delta$ and $d_G(w) \geq \Delta - 1$ and so 
 \[ d_w' \geq d_G(w)-2\eta\Delta-2 \geq (1-\gamma^2/2)\Delta\,. \]
 It follows that 
\[
\P\big( w \in G' \,|\, u, v \in G' \big)=\left(1-\frac{\gamma}{\Delta}\right)^{d_w}\leq 1- \frac{\gamma}{\Delta}d_w +\frac{\gamma^2 }{2\Delta^2}d_w^2\leq 1-\gamma  + \gamma^2\, .
\] Finally, summing over $w\in N_G(u) \cap N_G(v)$ we get \begin{equation*}\E \big( d_{G'}(u,v)\,|\,u,v\in G' \big) \leq (1 - \gamma + \gamma^2) d_G(u,v) \, . \qedhere \end{equation*}
\end{proof}

\subsection{Concentration and the proof of Lemma~\ref{lem:degs}}

We now bound the upper tails of $d_{G'}(v)$ and $d_{G'}(u,v)$ 
for $u,v \in G'$. We do this by using a concentration inequality of Chung and Lu \cite{chung_lu}, which is a one-sided variant of Freedman's inequality. 
For a martingale $(S_m)_{m=0}^M$, its \emph{increments} are $\xi_m = S_m-S_{m-1}$.

\begin{theorem}[Theorem 6.2 in~\cite{chung_lu}]\label{lem:Freedman}
Let $(S_m)_{m=0}^M$ be a martingale with respect to a filtration $(\mathcal{F}_m)_{m=0}^M$ and with increments $(\xi_i)_{i=1}^M$ for which $\xi_i\leq R_i$ and $\mathbb{E}[\, |\xi_i|^2 \, | \, \mathcal{F}_{i-1}]\leq \sigma_i^2$, almost surely. Set
\[
b= \sum_{i=1}^M( \sigma_i^2+R_i^2)\, .
\]
Then for all $r \geq 0$,
\[
\P(S_M-S_0 \geq r) \leq \exp\left(-\frac{r^2}{2b} \right)\, .
\]
\end{theorem}

For $u\in G$ we will choose a bipartite graph $H$ with bipartition $X\cup Y$ in such a way that $d_{G'}(u)=|X\setminus N_H(A)|$. For $v,w\in G$ we express $d_{G'}(v,w)$ similarly. With this in mind, we prove the following concentration inequality which will imply both \eqref{eq:concentration-degs} and \eqref{eq:concentration-codegs}.

\begin{lemma}\label{lem:abstract-concentration-lemma}
Let $0<p\leq 1/2$ and $H$ be a bipartite graph with bipartition $X \cup Y$ for which $d_H(x,x') \leq \ell $ for each $x,x' \in X$. Let $A$ be a $p$-random subset of $Y$ and let $S = |X\setminus N_H(A)|$. Then, for all $r \geq 0$,
\[  \P\big( S - \E\, [S] \geq r   \big) \leq \exp\left( -\frac{ r^2 }{ 4p( e_H(X,Y) + \ell|X|^2) }\right). \]
\end{lemma}
\begin{proof}
Write $Y = \{y_1,\ldots,y_M\}$ and let $\cF_t$ be the $\sigma$-algebra generated by $A \cap \{y_1,\ldots,y_t\}$. We now define the martingale
$S_t=\E[ S \,|\, \cF_t]$ which has increments $\xi_t =  \E[ S \, | \, \cF_t] - \E[S \, | \, \cF_{t-1}]$. We write $X' = X\setminus N(A)$ and express
\begin{equation}\label{eq:increments} \xi_t =  \sum_{x \in X } \Big( \P(x \in X' \,|\, \mathcal{F}_t) - \P(x \in X' \,|\, \mathcal{F}_{t-1}) \Big) \hspace{1mm} .  \end{equation}
Seeking to apply Theorem~\ref{lem:Freedman}, we will establish the following bounds on the increments $\xi_t$,
\begin{equation}\label{eq:increment-bounds} -d(y_t)\leq \xi_t \leq p\cdot d(y_t) \qquad  \text{ and } \qquad  y_t \not\in A  \,\Rightarrow  \, |\xi_t| \leq p \cdot d(y_t)\, . 
  \end{equation}
To this end, let $A_t = A \cap \{y_1,\ldots,y_t\}$ and, for $x \in X$, let $d_t(x) = |N(x) \cap \{y_{t+1},\ldots,y_m\}|$. We now note that, for all $x \in X$, we have 
\begin{equation}\label{eq:formula-for-conditional} \P( x \in X'\, \, | \,\, \mathcal{F}_t ) = (1-p)^{d_t(x)}\1( x \not\in N(A_t) )\, . \end{equation}
From this, we see that if $x \not\in N(y_t)$ then $\P(x \in X' \,|\, \mathcal{F}_t) = \P(x \in X' \,|\, \mathcal{F}_{t-1})$ and so \eqref{eq:increments} becomes 
\begin{equation} \label{eq:increments-nbd}
     \xi_t =  \sum_{x \in N(y_t) }\Big(\P(x \in X' \,| \mathcal{F}_t) - \P(x \in X' \,| \mathcal{F}_{t-1})\Big) \,.
\end{equation}
Combine \eqref{eq:formula-for-conditional} and \eqref{eq:increments-nbd} to write 
\begin{equation}\label{eq:utnotinA_t}
 \xi_t = \sum_{x \in N(y_t)} (1-p)^{d_t(x)}\1( x \not\in N(A_{t}) )-(1-p)^{d_t(x)+1}\1( x \not\in N(A_{t-1}) )   \, .
 \end{equation}
 If $y_t\not\in A$ then $\1( x \not\in N(A_{t}) )=\1( x \not\in N(A_{t-1}) )$ so $$0\leq \xi_t\leq \sum_{x \in N(y_t)} (1-p)^{d_t(x)}-(1-p)^{d_t(x)+1}\leq pd(y_t)\, .$$ On the other hand, if $y_t\in A$ then $$0\geq \xi_t\geq -\sum_{x \in N(y_t)} (1-p)^{d_t(x)+1}\geq -d(y_t)\, .$$       
This establishes~\eqref{eq:increment-bounds}. We now use the inequalities in \eqref{eq:increment-bounds} to bound $\xi_t\leq R_t:= pd(y_t)$ and $$\E[\, |\xi_t|^2 \,|\, \cF_{t-1}] \leq \sigma_t^2:= p d(y_t)^2+p^2 d(y_t)^2 \, .$$ Now we may bound the quantity $b$ that appears in Lemma~\ref{lem:Freedman}.
\[ b=\sum_{t = 1}^M \sigma_t^2+R_t^2 \leq \sum_{t = 1}^M  p d(y_t)^2 + 2p^2 d(y_t)^2  \leq 2p\sum_{t = 1}^M d(y_t)^2\, .
\]
We now observe that
\[
\sum_{t = 1}^M d(y_t)^2 = \sum_{t = 1}^M \sum_{x,x'\in X}\one[x, x'\in N_G(y_t)]
=\sum_{x,x'\in X} d_H(x,x')
\leq e(X,Y)+\ell|X|^2\, ,
\]
where the second equality comes from interchanging the sum and the inequality follows by considering the cases $x\neq x'$, $x=x'$ separately. We conclude that $b\leq 2p(e(X,Y)+\ell|X|^2)$.
    Thus, applying Theorem~\ref{lem:Freedman} yields our desired bound.\end{proof}

\begin{proof}[Proof of Lemma~\ref{lem:degs}]
To prove both \eqref{eq:concentration-degs} and \eqref{eq:concentration-codegs} we apply Lemma~\ref{lem:abstract-concentration-lemma}. For \eqref{eq:concentration-degs}, we consider the bipartite graph between $X = N_G(v)$ and $Y = N_G(X) \setminus (N_G(v)\cup\{v\})$. We use that  
\[ |X|\leq \Delta, \quad e_G(X,Y) \leq \Delta^2, \quad \text{ and } \hspace{5mm} \E\big(\hspace{.2mm} d_{G'}(v)\, \vert\, v \in G'\hspace{.2mm} \big) \leq \left(1-\gamma+\gamma^2\right)d_G(v). \]
The first two inequalities hold by the maximum degree condition on $G$ and the third holds by Lemma~\ref{lem:E-deg}. Applying Lemma~\ref{lem:abstract-concentration-lemma} with $p=\gamma/\Delta, \ell=\eta\Delta$ and $r = (\alpha-\gamma^2) \Delta\geq \alpha \Delta/2$ gives 
\[  \P\big( d_{G'}(v) > \left(1-\gamma+\alpha\right)d_G(v)\, \vert\,  v \in G' \big) \leq \exp\left( \frac{ - r^2 }{ 4p( e(X,Y) + \ell|X|^2) }\right) \leq \exp\left(-\frac{\alpha^2}{{32 \eta\gamma}} \right), \]
using that $e(X,Y)\leq \Delta^2\leq \eta \Delta^3$ thus proving \eqref{eq:concentration-degs}. 

To deal with the concentration of \emph{codegrees} we consider the bipartite graph with $X = N_G(u) \cap N_G(v)$ and $Y = N_G(X) \setminus (N_G(u)\cup N_G(v)\cup\{u,v\})$. We note that 
\[|X|\leq \eta\Delta,\quad e_G(X,Y) \leq \eta \Delta^2, \quad \text{ and } \quad \E\big(\hspace{.2mm} d_{G'}(u,v)\, \vert\, u, v \in G'\hspace{.2mm} \big)  \leq \left(1-\gamma+\gamma^2\right)d_G(u,v), \] where the first two inequalities follow from the maximum degree and codegree conditions on $G$ and the third holds by Lemma~\ref{lem:E-deg}. Thus applying Lemma~\ref{lem:abstract-concentration-lemma} with $p=\gamma/\Delta, \ell  = \eta \Delta$ and $r = (\alpha - \gamma^2)\eta \Delta \geq \alpha \eta \Delta / 2$ gives
\begin{align*}
\P\big(d_{G'}(u,v) \geq  \left(1-\gamma+\alpha\right)\eta \Delta \, \vert\, u,v \in G' \big) &\leq  \exp\left(-  \frac{(\alpha \eta \Delta / 2)^2}{4(\gamma/\Delta)(\eta \Delta^2+\eta^3\Delta^3)}\right)
\end{align*}   
which is at most $\exp\left(- \frac{\alpha^2}{32 \gamma \eta}\right)$, using that $\eta^2\Delta\geq 1$. This completes the proof.\end{proof}

\section{Proof of Theorem~\ref{thm:ind-graph} }
In this section we prove Theorem~\ref{thm:ind-graph} in three steps. We first define and prove that we can take our desired nibble out of our graph $G$. After this, we show that we can \emph{add} edges to the resulting graph to make it (nearly) regular,  without affecting the codegrees. In the final part of the section, 
we prove Theorem~\ref{thm:ind-graph} by alternatively applying these two steps 
to find a large independent set in our given graph. 

\subsection{Nibble step}
The following is  our ``nibble'' step, formally spelled out. The idea is that given a graph $G$, we will ``nibble'' a set $A \subset V(G)$ that is nearly an independent set; when we remove $A \cup N(A)$ from $G$, we can remove a small number of vertices to obtain a set $C$ whose maximum degree and co-degree have shrunk by an appropriate amount.

\begin{lemma}\label{lem:key}
For $\Delta \geq 2^{11}$, let $8\Delta^{-1/8}\leq \gamma \leq 1/2$ and $n\geq \Delta^4$. Now let $G$
be a graph on $n$ vertices for which 
\[ d(v) \in \{\Delta,\Delta-1\} \quad \text{ and }
 \quad d(u,v) \leq 2^{-6}\gamma^3\Delta (\log \Delta )^{-1},\]
for all $u,v \in V(G)$. Then for any $\alpha\geq 2\gamma^2$ there exists 
\[ A\subset V(G) \qquad \text{  and  } \qquad C\subset V(G)\setminus (A\cup N_G(A)) \] so that $A$ satisfies 
\[|A|\geq(1-\alpha)\gamma n/\Delta, \qquad e(G[A])\leq \gamma^2 n/\Delta
\]
and $C$ satisfies
$$ |C|\geq (1-\gamma-\alpha)n, \quad \Delta(G[C]) \leq (1-\gamma+\alpha)\Delta, \quad \Delta_2(G[C])\leq  (1-\gamma+\alpha)\max\{\Delta_2(G),2\sqrt{\Delta}\}.$$
\end{lemma}

\begin{proof}
Take $A$ to be a $p$-random subset of $V(G)$ with $p = \gamma/\Delta$ and notice that, for example by Chebyshev's inequality,
\[ |A|\geq (1-\alpha)\gamma n/\Delta, \quad e(G[A])\leq  \gamma^2 n/\Delta \quad  \text{ and } \quad |A\cup N_G(A)|\leq (\gamma+\alpha/2)n,\] with probability at least $1/2$. Let $G'=G\setminus  (A\cup N_G(A))$ and let $B \subseteq V(G')$ be 
\[B = \big\{ u : d_{G'}(u)\geq (1-\gamma+\alpha)\Delta \big\} \cup \big\{ u : \exists v\in V(G') \text{ s.t. } d_{G'}(u,v)\geq (1-\gamma+\alpha)\eta\Delta\, \big\} \]
where we define $\eta = \max\{\Delta_2(G),2\sqrt{\Delta}\}/\Delta$.  We now set $C = V(G')\setminus B$ so that the second and third conclusions pertaining to the set $C$ are automatically satisfied. Thus it only remains to show that we can take $|B| \leq \alpha n/2$. For this we prove
\begin{equation}\label{eq:B-exp} \E\, |B|\leq \alpha n/8\end{equation}
and thus $|B|\leq \alpha n/2$ with probability at least $3/4$ by Markov's inequality. 

Now we note that our choice of parameters satisfy \eqref{eq:param-assumptions} so we may apply Lemma~\ref{lem:degs}. Therefore we have that  \begin{equation}\label{eq:lemdegs-app}
    \P\big( d_{G'}(u)\geq (1-\gamma+\alpha)\Delta\,  \vert\, u \in G' \big)\leq \exp\left(-\frac{\alpha^2}{32\gamma \eta}\right)\, .
\end{equation}
If $\eta=\Delta_2/\Delta \leq 2^{-6}\gamma^3(\log \Delta)^{-1}$ then since $\alpha\geq 2\gamma^2$ we have $\alpha^2/(32\gamma \eta)\geq \gamma^3/(8\eta) \geq  8\log\Delta$. If on the other hand $\eta=2\Delta^{-1/2}$, then since $\gamma\geq 8\Delta^{-1/8}$ we have $\alpha^2/(32\gamma \eta)\geq \gamma^3/(8\eta)\geq 2^6\Delta^{1/8}$. In either case the right hand side of~\eqref{eq:lemdegs-app} is at most $\alpha/16$.


If $\Delta_2(G)\leq \sqrt{\Delta}$ then $\P\big(\exists v: d_{G'}(u,v)\geq (1-\gamma+\alpha)\eta \Delta \,|\, u \in G'\big)=0$, since $\eta\Delta = 2\sqrt{\Delta}$ so we may assume $\Delta_2(G)>\sqrt{\Delta}$ and consequently 
\begin{align}\label{eq:etaEst}
\eta\leq 2\Delta_2(G)/\Delta\leq 2^{-5}\gamma^3(\log \Delta)^{-1}\, .
\end{align}
Finally we note that
\[ \P\big(\exists v: d_{G'}(u,v)\geq (1-\gamma+\alpha)\eta \Delta \,|\, u \in G'\big) \leq \sum_{v\in N_G(N_{G}(u))}\P\big( d_{G'}(u,v)\geq (1-\gamma+\alpha)\eta \Delta  \,|\,  u,v \in G'\big)  \nonumber \]
which, using Lemma~\ref{lem:degs} and~\eqref{eq:etaEst}, is at most
\begin{equation} \label{eq:lemcodegs-app} \Delta^2 \cdot \max_{u,v}\, \P\big( d_{G'}(u,v)\geq (1-\gamma+\alpha)\eta \Delta \,|\, u,v \in G'\big)   \leq \Delta^2 \exp\left(-\frac{\alpha^2}{32\gamma\eta}\right)   \leq \alpha/16\,. \end{equation}
Thus summing over \eqref{eq:lemdegs-app} and \eqref{eq:lemcodegs-app} establishes \eqref{eq:B-exp}. This completes the proof.
\end{proof}

\subsection{Regularization step}
We now show that we can add edges to make our graph almost regular. The proof is a simple greedy algorithm where we add an edge as long at it doesn't make the maximum degree or codegree of our initial graph increase by more than 1.

\begin{lemma}\label{lem:add_edges}
For $\Delta\geq 2$, let $n \geq 2\Delta^4$ and let $G$ be a graph on $n$ vertices with  $\Delta(G)\leq \Delta$. Then there exists a graph $\Gbar$ with $V(\Gbar)= V(G)$ and $E(\Gbar)\supset E(G)$ so that 
\[ d_{\Gbar}(v) \in \{ \Delta, \Delta+1 \} \qquad \text{ and } \qquad \Delta_2(\Gbar)\leq \max\{\Delta_2(G),1\}.\]
\end{lemma}
\begin{proof}
We construct $\Gbar \supseteq G$ by forming a sequence of graphs $G_{i+1} = G_i + e_{i+1}$, where $G_0 = G$. We add edges in two phases: $i \in [T]$ and $i \in [T+1, T']$ for $T$ and $T'$ to be defined later. For this, we define 
\[ D_i = \{v \in V(G_i) : d_{G_i}(v) < \Delta \}, \quad \text{ and } \quad S_i = \{ v \in V(G_i) : d_{G_i}(v) \leq \Delta \},\] the set of ``deficient'' vertices and the set of vertices that we can still add an edge to. 

In the first phase, $i \in [T]$, we choose $u,v \in D_i$ 
for which $\dist_{G_i}(u,v) \geq 4$, if such a pair exists, and then fix $e_{i+1} = \{u,v\}$. If no such pair exists we conclude this first phase with the graph $G_T$. For the second phase, we choose $u \in D_i$ and 
$v \in S_i$ with $\dist_{G_i}(u,v) \geq 4$ and then set $e_{i+1} = \{u,v\}$.
If no such pair exists, we stop this second phase with the graph $G_{T'}$. We then set $\Gbar = G_{T'}$, which is our final graph.

We claim that $\Gbar$ has no vertices of degree $<\Delta$; in particular, we claim that the second phase terminates due to having $D_{T'} = \emptyset$.   

At the end of the first phase, we have $\dist_{G_T}(u,v) \leq 3$, for all $u,v \in D_T$. Since $\Delta(G_T) \leq \Delta$, this implies that $|D_T| \leq \Delta^3$. By the definition of the  second phase, we have
\[ |S_{i+1}| \geq |S_{i}|-1 \quad \text{ and } \quad \sum_{v\in D_{i+1}}(\Delta - d_{G_{i+1}}(v)) \leq \left(\sum_{v\in D_{i}}(\Delta - d_{G_{i}}(v))\right)-1,\]
for $i\geq T$.  Since $|D_T| \leq \Delta^3$ we have $$\left(\sum_{v\in D_{T}}(\Delta - d_{G_{T}}(v))\right) \leq \Delta^4$$
and so it will take at most $\Delta^4$ steps of the second phase until we have $D_{T'} = \emptyset$ and so $T' - T \leq \Delta^4$.  Note that since $|S_T| = n \geq 2\Delta^4$ we thus have $|S_i| \geq \Delta^4$ for each $i \in [T,T']$.  This implies that for each $v \in D_{T}$ and $i \in [T,T']$ we have that $|S_{i} \setminus N^{(3)}(v)| \geq |S_i| -  \Delta^3 > 0$ where $N^{(3)}(v)$ denotes the set of all vertices whose distance to $v$ is at most $3$.  In particular, at each step in the second phase, if $D_i \neq \emptyset$, then for each $v \in D_i$ there is some $u \in S_{i}$ with $\dist_{G_i}(u,v) \geq 4$.  This proves $D_{T'} = \emptyset$ and in particular all vertices of $\Gbar$ have degree $\Delta$ or $\Delta+1$. 

Finally, we claim that for all distinct $u,v \in V(\Gbar)$, we have $d_{\Gbar}(u,v)\leq \max\{1,d_{G}(u,v)\}$. To see this note that if $d_{G_{i+1}}(u,v)\geq
d_{G_{i}}(u,v)+1$ and $d_{G_i}(u,v)>0$ then $e_{i+1} = \{u',v'\}$ must be such that 
$\dist_{G_i}(u',v') \leq 3$, which contradicts our choice of $e_{i+1}$. 
\end{proof}

\subsection{The proof of Theorem~\ref{thm:ind-graph}}

Given $0<\eps<1/2$, let 
\[ \Delta\geq \exp\left((80/\eps)\log (80/\eps)\right) \quad \text{ so that } \quad \eps\geq \frac{40(\log \log \Delta +2) }{\log \Delta}\, .\]
We prove Theorem~\ref{thm:ind-graph} by iteratively applying the regularization step (Lemma~\ref{lem:add_edges}) followed by the nibble step (Lemma~\ref{lem:key}). To set up our iteration we fix $\gamma=(\log \Delta)^{-2}$, $\alpha=2\gamma^2$ and $q = 1-\g+2\alpha$. We define 
\begin{equation}
\Delta_i = \lceil q^i(\Delta(G)+1)\rceil \qquad \text{ and } \qquad \Delta'_i = q^i\Delta_2(G). 
\end{equation}
We will construct a sequence of graphs $G = G_0, \Gbar_0, G_1, \Gbar_1, \cdots  ,G_T$ with 
\[ V(G_0) = V(\Gbar_0) \supset V(G_1) = V(\Gbar_1) \supset \cdots \supset V(G_T) \not= \emptyset ,\]
where $\Gbar_i$ is obtained from $G_i$ by applying our regularization step, Lemma~\ref{lem:add_edges}, with parameter $\Delta_i-1$, and $G_{i+1} = G_{i}[C_{i+1}]$, where $C_{i+1} = C$ comes from Lemma~\ref{lem:key} applied to $\Gbar_i$ with parameter $\Delta_i$.  We note that since $\Gbar_i$ consists of $G_i$ with additional edges, then each independent set of $\Gbar_i$ corresponds to an independent set in $G_i$. For each step $i$ we will ensure 
\begin{equation}\label{eq:ind-degree-cond} \Delta(G_i) \leq \Delta_i-1 \quad \text{ and }\quad  d_{\Gbar_i}(v) \in \{ \Delta_i-1,\Delta_i \}, \text{ for all } v \in \Gbar_i \end{equation}
and that, in each step,
\begin{equation}\label{eq:ind-codegree-cond} \Delta_2(G_i), \Delta_2(\Gbar_i) \leq \max\{ \Delta'_i,2\sqrt{\Delta_i}\}\,.
\end{equation}
Maintaining these conditions will allow us to ensure 
\begin{equation}\label{eq:ind-G_i-size} |G_i| \geq (q-3\alpha)^i n \end{equation}
which allows us to iterate our process for $T$ steps where
\begin{equation} \label{eq:def-T}
T= \gamma^{-1}\big(\log \Delta -32(\log\log\Delta+2)\big)\, .
\end{equation} 
This choice of $T$ ensures that $\Delta_i\geq (4 \log \Delta)^{16}$ throughout the process.

To build our independent set, note that from Lemma~\ref{lem:key} we obtain a sequence of disjoint vertex sets $A_1,\ldots, A_T$ with the property that there are no edges between  $A_i,A_j$ in $G$, for $i\not= j$. To finish, we take maximum independent sets
$I_i \subseteq A_i$ and define $I = I_1 \cup \cdots \cup I_T$. We show that for each $i \in [T]$ we have 
\begin{equation} \label{eq:ind-A_iandI_i}  
|I_i| \geq (1-8\gamma^{1/2})\gamma n/\Delta \end{equation}
so that 
\[|I| \geq T (1-8\gamma^{1/2})\gamma n/\Delta \geq (1-\eps)\frac{n\log \Delta}{\Delta},   \] as desired. 
We now formally prove Theorem~\ref{thm:ind-graph}.
\begin{proof}[Proof of Theorem~\ref{thm:ind-graph}]
First note that if $H$ is the union of disjoint copies of $G$ then $\frac{\alpha(G)}{|G|}=\frac{\alpha(H)}{|H|}$. Therefore by taking sufficiently many disjoint copies of our initial graph, we may assume without loss of generality that 
\begin{align}\label{eq:DiSjointWlog}
(q-3\alpha)^T n\geq 2(\Delta+1)^4\, .
\end{align}
 We begin by proving the statements  \eqref{eq:ind-degree-cond}, \eqref{eq:ind-codegree-cond} and \eqref{eq:ind-G_i-size} hold for $G_i$ by induction on $i \in [T]$. Throughout we write $\Delta = \Delta(G)$.

For $i = 0 $ we note that the three statements trivially hold for $G_0$ (not necessarily $\Gbar_0$, yet). Now assuming that they hold for $G_i$, we look to apply Lemma~\ref{lem:add_edges} with parameter $\Delta_i-1$. For this we need only that $|G_i| \geq 2\Delta_i^4$, which follows from the inductive assumption $|G_i| \geq (q-3\alpha)^i n$ and~\eqref{eq:DiSjointWlog}.
Thus we obtain a graph $\Gbar_i$ for which $d_{\Gbar_i}(v) \in \{\Delta_i-1,\Delta_i\}$ and 
$\Delta_2(\Gbar_i) \leq \max\{\Delta_2(G_i),1\}$. Since~\eqref{eq:ind-codegree-cond} holds for $G_i$, this shows $\Gbar_i$ satisfies \eqref{eq:ind-degree-cond} and \eqref{eq:ind-codegree-cond}.

We now look to apply Lemma~\ref{lem:key} to $\Gbar_i$.  
By~\eqref{eq:ind-codegree-cond} we know $\Delta_2(\Gbar_i)\leq \max\{q^i\Delta_2(G),2\sqrt{\Delta_i}\}$ and thus we check two cases. If $q^i\Delta_2(G)\geq 2\sqrt{\Delta_i}$ then 
\[\Delta_2(\Gbar_i)\leq q^i \Delta_2(G) = (q^i\Delta) (\Delta_2(G)/\Delta) \leq 2^{-6}\Delta_i \cdot  \gamma^3 (\log \Delta_i)^{-1} \, ,\] since $q^i\Delta\leq \Delta_i$ and $\Delta_2(G)\leq \Delta (2 \log \Delta)^{-7}$.
On the other hand if $2\sqrt{\Delta_i}\geq q^i\Delta_2(G)$ then 
\[\Delta_2(\Gbar_i)\leq 2\sqrt{\Delta_i}  = 2 \Delta_i (\Delta_i)^{-1/2} \leq 2 \Delta_i (2 \log \Delta)^{-7} \leq 2^{-6}\Delta_i \cdot  \gamma^3 (\log \Delta_i)^{-1}\] 
where we used that $\Delta_i \geq \Delta_T \geq (4 \log \Delta)^{16}$.   This lower bound on $\Delta_i$ also shows that $$\gamma = (\log \Delta)^{-2}\geq 8\Delta_i^{-1/8}$$ for all $i$. 
Thus we can apply Lemma~\ref{lem:key} to obtain a graph $G_{i+1} = G_{i}[C_{i+1}]$ for which 
\[ |G_{i+1}|\geq (1-\gamma-\alpha)|G_i|, \quad \Delta(G_{i+1}) \leq (1-\gamma+\alpha)\Delta_i, \quad \Delta_2(G_{i+1}) \leq (1-\gamma+\alpha) \max\{\Delta_i',2\sqrt{\Delta_i}\}.\]
Thus we have established \eqref{eq:ind-degree-cond}, \eqref{eq:ind-codegree-cond} and \eqref{eq:ind-G_i-size}, by induction. 

We now demonstrate that we also have \eqref{eq:ind-A_iandI_i}. Indeed in each step $i \in [T]$, by our application of Lemma~\ref{lem:key}, we obtained  $A_i$ satisfying 
\[ |A_i| \geq (1-\alpha)\gamma|G_i|/\Delta_i \qquad \text{ and } \qquad  e(G[A_i])\leq \gamma^2 |G_i|/\Delta_i\,.   \]
We see that at most a $2\gamma$ fraction of  vertices in $A_i$ have non-zero degree; deleting vertices with non-zero degree gives an independent set $I_i \subset A_i$ for which 
\[ |I_i| \geq (1-2\gamma)|A_i|\geq (1-2\gamma)(1-\alpha) \gamma |G_i|/\Delta_i \, .\] Now using the lower bound on $|G_i|$ in \eqref{eq:ind-G_i-size} and the definition of $\Delta_i$ we get \[ |I_i|\geq (1-2\gamma)(1-3\alpha)^{i+1} \gamma n/\Delta \, .\]
Since $i \leq T$ and $\alpha=2(\log \Delta)^{-4}$ we see that $(1 - 3\alpha)^{i+1} \geq (1 - 7 \gamma^{1/2})$.  This establishes \eqref{eq:ind-A_iandI_i} and therefore Theorem~\ref{thm:ind-graph}. 
\end{proof}

\section*{Acknowledgments}
The authors thank Rob Morris for comments on a previous draft. Matthew Jenssen is supported by a UKRI Future Leaders Fellowship MR/W007320/1.  Marcus Michelen is supported in part by NSF grants DMS-2137623 and DMS-2246624.

\appendix

 \section{Proof of Theorem~\ref{thm:sphere-codes}}
The only modification needed to adapt the proof of Theorem~\ref{thm:main} to prove Theorem~\ref{thm:sphere-codes} is in the proof of the  ``discretisation step'' Lemma~\ref{prop:preprocessing}. First let us introduce some notation.

Throughout this section we fix $\theta\in (0,\pi/2)$. For a measurable set $A \subseteq S^{d-1}$, let $s(A)$ denote the normalized surface area of $A$ so that $s(S^{d-1}) =1$.
For $x \in S^{d-1}$, let 
\[ C_{x}(\theta) = \{ y \in S_{d-1}: \langle x, y \rangle \ge \cos \theta \} \qquad \text{ and }  \qquad s_d(\theta) = s(C_{x}(\theta)). \]
For a finite $X \subseteq S^{d-1}$ we define the graph $G(X,\theta)$ to be the graph on vertex set $X$ where two distinct $x,y \in X$ are connected whenever $\langle x,y \rangle \geq \cos \theta$.
Thus an independent set in $G(X,\theta)$ is a spherical code of angle $\theta$. We have the following analogue of Lemma~\ref{prop:preprocessing}.

\begin{lemma}\label{prop:preprocessingsphere}

    For all $d$ sufficiently large, there exists $X \subset S^{d-1}$ so  that 
\[ |X| \geq   \big(1-1/d\big) \Delta (s_d(\theta))^{-1}, \quad \text{ where } \quad \Delta = s_d(\theta)\left(\frac{\sqrt{d}}{2\log d}\right)^d, \]
and if $G = G(X,\theta)$, for all $v \in X$ we have 
\begin{equation*} 
    d_G(v) \leq \Delta\big(1 + \Delta^{-1/3}  \big) \qquad  \text{ and } \qquad  d_G(u,v) \leq 2\Delta  \cdot e^{-c(\log d)^2 }, \end{equation*}
for each distinct $u,v \in V(G)$ where $c = c_\theta>0$ is a constant depending only on $\theta$.\end{lemma}

If we assume Lemma~\ref{prop:preprocessingsphere}, Theorem~\ref{thm:sphere-codes} follows quickly. 

\begin{proof}[Proof of Theorem~\ref{thm:sphere-codes}]
    Let ${X}$ and $G$ be as in Lemma~\ref{prop:preprocessingsphere}. Apply Theorem~\ref{thm:ind-graph} to find $I \subseteq X$, which is an independent set in $G$ such that 
\[ |I| \geq  (1-o(1))\frac{|X|\log \Delta }{\Delta} \geq (1-o(1)) \frac{d\log d}{2s_d(\theta)}\, .\]
Note that $I$ corresponds to a spherical code of angle $\theta$. 
\end{proof}

To prove Lemma~\ref{prop:preprocessingsphere}, we set $\lambda =\left(\frac{\sqrt{d}}{2\log d}\right)^d$ and write $\bX\sim \mathrm{Po}_\lambda(S^{d-1})$ to denote a Poisson process of intensity $\lambda$ on $S^{d-1}$ with $s(\cdot)$ as the underlying measure. We first prove an analogue of Fact~ \ref{fact:sphere-intersection}. For this we use the following formula for the area of a spherical cap.

\begin{align}\label{eq:CapArea}
s_d(\theta)&= \frac{1}{\sqrt{\pi}}  \frac{ \Gamma(d/2)}{ \Gamma((d-1)/2)}\int_{0}^\theta \sin^{d-2} x  \, dx \, =  (1+o(1)) \frac{1}{ \sqrt{2 \pi d}} \cdot \frac{\sin^{d-1} \theta}{\cos \theta}\, ,
\end{align}
where $o(1)$ tends to zero as $d\rightarrow \infty$.

\begin{fact}\label{fact:sphere-intersectionsphere}
Let $\tau\in (0,2\theta)$ and let $x,y\in S^{d-1}$ be such that $\langle x, y \rangle=\cos(\tau)$. Then
\[
s(C_{\theta}(x)\cap C_{\theta}(y))\leq (1+o(1))s_d(\theta)e^{-c(\theta) \cdot \tau^2d}\, ,
\]
where $c(\theta)=(\cot^2\theta)/16$.
\end{fact}
\begin{proof}
Without loss of generality we may assume $x=(1,0,\ldots, 0)$ and $y=(\cos(\tau), \sin(\tau),0,\ldots,0)$. Suppose now that $z\in C_{\theta}(x)\cap C_{\theta}(y)$ i.e.
$\langle x, z \rangle\geq\cos(\theta)$ and $\langle y, z \rangle\geq\cos(\theta)$. 
Let $m$ be the `spherical midpoint' between $x$ and $y$, that is, $m=(\cos(\tau/2), \sin(\tau/2),0,\ldots,0)$. 
A simple calculation shows that 
\[
\langle z,m \rangle \geq \frac{\cos \theta}{\cos(\tau/2)}\, .
\]
It follows that 
$C_{\theta}(x)\cap C_{\theta}(y)\subseteq C_{\rho}(m)$
where 
\[
\rho=\arccos\left(
\frac{\cos \theta}{\cos(\tau/2)}
\right)\, .
\]
In particular, $s(C_{\theta}(x)\cap C_{\theta}(y))\leq s_d(\rho)$.
We now estimate $s_d(\rho)$. By~\eqref{eq:CapArea}
\begin{align}\label{eq:CapRatio}
\frac{s_d(\rho)}{s_d(\theta)}=(1+o(1)) \left(\frac{\sin \rho}{\sin \theta}\right)^{d-1}  \frac{\cos \theta}{\cos \rho}\, .
\end{align}
Now,
\[
\sin^2 \rho =  1-\frac{\cos^2\theta}{\cos^2(\tau/2)} \leq 1- (1+\tau^2/4)\cos^2\theta = \sin^2\theta - (\tau^2/4)\cos^2\theta\, ,
\]
where we used the inequality $1/\cos^2x\geq 1+x^2$ for all $x\in (0,\pi/2)$.
Returning to~\eqref{eq:CapRatio} we conclude that
\[
\frac{s_d(\rho)}{s_d(\theta)}\leq(1+o(1)) \left(1-(\tau^2/4)\cot^2\theta\right)^{(d-1)/2} \cos(\tau/2)\, .
\]
The result follows.
\end{proof}

The following two lemmas now follow in a near-identical way to their counterparts in Section~\ref{sec:preprocessing}.

\begin{lemma}\label{lem:prec-bad-1sphere} Let $\bX \sim \mathrm{Po}_\lambda(S^{d-1})$. If $d$ is sufficiently large, then 
    $$\E\left|\big\{x \in \bX :  |\bX \cap C_{x}(\theta)| \geq \Delta\big(1 + \Delta^{-1/3}\big) \big\}   \right| \leq (2d)^{-1} \cdot \E|\bX| \, .  $$
\end{lemma}
 
Let $c(\theta)=(\cot^2\theta)/16$ as in Fact~\ref{fact:sphere-intersectionsphere}.  Applying Fact~\ref{fact:sphere-intersectionsphere} in the place of Fact~\ref{fact:sphere-intersection} in the proof of Lemma~\ref{lem:prec-bad-2} shows the following:
\begin{lemma}\label{lem:prec-bad-2sphere} 
    Let $\bX \sim \mathrm{Po}_\lambda\big( S^{d-1} \big)$ and put $\eta = 2e^{-c(\theta)(\log d)^2}$. If $d$ is sufficiently large, then
    $$\E\, \big|\big\{x \in \bX : \exists \hspace{.5mm} y \in \bX : |\bX \cap C_{x}(\theta) \cap C_{y}(\theta)| \geq  \eta \Delta  \big\}   \big| \leq (2d)^{-1} \cdot \E |\bX | \,.  $$
\end{lemma}

We can now prove Theorem~\ref{prop:preprocessingsphere} by simply sampling $\bX \sim \mathrm{Po}_\lambda(S^{d-1})$ and deleting problematic points. 

\begin{proof}[Proof of Lemma~\ref{prop:preprocessingsphere}]
    Let $\bX\sim \mathrm{Po}_\lambda(S^{d-1})$. Let $B$ be the points $x\in \bX$ that satisfy
    \begin{equation*} |\bX \cap C_x(\theta)| \geq \Delta\big(1 + \Delta^{-1/3}\big) \quad \text{ or } \quad |\bX \cap C_x(\theta) \cap C_y(\theta)| \geq \Delta \cdot 2e^{-c_\theta(\log d)^2}\, ,\end{equation*}
    for some $y\in\bX$.
    We set $X =  \bX \setminus B$ and apply Lemma~\ref{lem:prec-bad-1sphere} and Lemma~\ref{lem:prec-bad-2sphere} to see  \[ \E|X| \geq \E |\bX| - \E |B| \geq (1-1/d) \cdot \E |\bX |. \]
 Noting that $\E |\bX| =\lambda= \Delta/s_d(\theta)$, finishes the proof.
\end{proof}

\bibliography{sphere}

\begin{thebibliography}{10}

\bibitem{ajtai1980note}
M.~Ajtai, J.~Koml\'os, and E.~Szemer\'edi.
\newblock A note on {R}amsey numbers.
\newblock {\em J. Combin. Theory Ser. A}, 29:354--360, 1980.

\bibitem{ajtai1981dense}
M.~Ajtai, J.~Koml\'{o}s, and E.~Szemer\'{e}di.
\newblock A dense infinite {S}idon sequence.
\newblock {\em European J. Combin.}, 2(1):1--11, 1981.

\bibitem{ball1992lower}
K.~Ball.
\newblock A lower bound for the optimal density of lattice packings.
\newblock {\em Int. Math. Res. Not.}, 1992(10):217--221, 1992.

\bibitem{BK-triangle-free}
T.~Bohman and P.~Keevash.
\newblock Dynamic concentration of the triangle-free process.
\newblock {\em Random Structures \& Algorithms}, 58(2):221--293, 2021.

\bibitem{charbonneau2021three}
P.~Charbonneau, P.~K. Morse, W.~Perkins, and F.~Zamponi.
\newblock Three simple scenarios for high-dimensional sphere packings.
\newblock {\em Phys. Rev. E}, 104(6):Paper No. 064612, 15, 2021.

\bibitem{chung_lu}
F.~Chung and L.~Lu.
\newblock Concentration inequalities and martingale inequalities: a survey.
\newblock {\em Internet Math.}, 3(1):79--127, 2006.

\bibitem{cohn2016packing}
H.~Cohn.
\newblock Packing, coding, and ground states.
\newblock {\em arXiv:1603.05202}, 2016.

\bibitem{cohn2016conceptual}
H.~Cohn.
\newblock A conceptual breakthrough in sphere packing.
\newblock {\em Notices Amer. Math. Soc.}, 64:102--115, 2017.

\bibitem{cohn2016sphere}
H.~Cohn, A.~Kumar, S.~D. Miller, D.~Radchenko, and M.~Viazovska.
\newblock The sphere packing problem in dimension 24.
\newblock {\em Ann. of Math.}, 185:1017--1033, 2017.

\bibitem{cohn2014sphere}
H.~Cohn and Y.~Zhao.
\newblock Sphere packing bounds via spherical codes.
\newblock {\em Duke Math. J.}, 163:1965--2002, 2014.

\bibitem{davenport1947hlawka}
H.~Davenport and C.~A. Rogers.
\newblock Hlawka's theorem in the geometry of numbers.
\newblock {\em Duke Math. J.}, 14:367--375, 1947.

\bibitem{davies2015independent}
E.~Davies, M.~Jenssen, W.~Perkins, and B.~Roberts.
\newblock Independent sets, matchings, and occupancy fractions.
\newblock {\em J. Lond. Math. Soc.}, 96(1):47--66, 2017.

\bibitem{davies2018average}
E.~Davies, M.~Jenssen, W.~Perkins, and B.~Roberts.
\newblock On the average size of independent sets in triangle-free graphs.
\newblock {\em Proc. Amer. Math. Soc.}, 146:111--124, 2018.

\bibitem{fernandez2021new}
I.~G. Fern{\'a}ndez, J.~Kim, H.~Liu, and O.~Pikhurko.
\newblock New lower bounds on kissing numbers and spherical codes in high
  dimensions.
\newblock {\em arXiv preprint arXiv:2111.01255}, 2021.

\bibitem{FGM-triangle-free}
G.~Fiz~Pontiveros, S.~Griffiths, and R.~Morris.
\newblock The triangle-free process and the {R}amsey number {$R(3,k)$}.
\newblock {\em Mem. Amer. Math. Soc.}, 263(1274):v+125, 2020.

\bibitem{frisch1987nonuniform}
H.~L. Frisch and J.~K. Percus.
\newblock Nonuniform classical fluid at high dimensionality.
\newblock {\em Phys. Rev. A}, 35(11):4696--4702, 1987.

\bibitem{frisch1999high}
H.~L. Frisch and J.~K. Percus.
\newblock High dimensionality as an organizing device for classical fluids.
\newblock {\em Phys. Rev. E}, 60(3):2942--2948, 1999.

\bibitem{frisch1985classical}
H.~L. Frisch, N.~Rivier, and D.~Wyler.
\newblock Classical hard-sphere fluid in infinitely many dimensions.
\newblock {\em Phys. Rev. Lett.}, 54(19):2061--2063, 1985.

\bibitem{groemer1963existenzsatze}
H.~Groemer.
\newblock Existenzs\"{a}tze f\"{u}r {L}agerungen im {E}uklidischen {R}aum.
\newblock {\em Math. Z.}, 81:260--278, 1963.

\bibitem{hales2005proof}
T.~C. Hales.
\newblock A proof of the {K}epler conjecture.
\newblock {\em Ann. of Math.}, 162(3):1065--1185, 2005.

\bibitem{janson2011random}
S.~Janson, T.~{\L}uczak, and A.~Rucinski.
\newblock {\em Random graphs}.
\newblock Wiley-Interscience Series in Discrete Mathematics and Optimization.
  Wiley-Interscience, New York, 2000.

\bibitem{jenssen2018kissing}
M.~Jenssen, F.~Joos, and W.~Perkins.
\newblock On kissing numbers and spherical codes in high dimensions.
\newblock {\em Adv. Math.}, 335:307--321, 2018.

\bibitem{jenssen2019hard}
M.~Jenssen, F.~Joos, and W.~Perkins.
\newblock On the hard sphere model and sphere packings in high dimensions.
\newblock {\em Forum Math. Sigma}, 7:Paper No. e1, 19, 2019.

\bibitem{kabatiansky1978bounds}
G.~A. {Kabatjanski\u\i} and V.~I. {Leven\v ste\u\i n}.
\newblock Bounds for packings on the sphere and in space.
\newblock {\em Problemy Pereda\v ci Informacii}, 14:3--25, 1978.

\bibitem{krivelevich2004lower}
M.~Krivelevich, S.~Litsyn, and A.~Vardy.
\newblock A lower bound on the density of sphere packings via graph theory.
\newblock {\em Int. Math. Res. Not.}, 2004:2271--2279, 2004.

\bibitem{last2017lectures}
G.~Last and M.~Penrose.
\newblock {\em Lectures on the Poisson process}, volume~7.
\newblock Cambridge University Press, 2017.

\bibitem{lowen2000fun}
H.~L\"{o}wen.
\newblock Fun with hard spheres.
\newblock In {\em Statistical physics and spatial statistics ({W}uppertal,
  1999)}, volume 554 of {\em Lecture Notes in Phys.}, pages 295--331. Springer,
  Berlin, 2000.

\bibitem{Min05}
H.~Minkowski.
\newblock Diskontinuit\"atsbereich f\"ur arithmetische {\"a}quivalenz.
\newblock {\em J. Reine Angew. Math.}, 129:220--274, 1905.

\bibitem{parisi2006amorphous}
G.~Parisi and F.~Zamponi.
\newblock Amorphous packings of hard spheres for large space dimension.
\newblock {\em J. Stat. Mech. Theory Exp.}, (3):P03017, 15, 2006.

\bibitem{parisi2010mean}
G.~Parisi and F.~Zamponi.
\newblock Mean-field theory of hard sphere glasses and jamming.
\newblock {\em Rev. Modern Phys.}, 82(1):789, 2010.

\bibitem{rankin1955closest}
R.~A. Rankin.
\newblock The closest packing of spherical caps in {$n$} dimensions.
\newblock {\em Proc. Glasgow Math. Assoc.}, 2:139--144, 1955.

\bibitem{rodl}
V.~R\"{o}dl.
\newblock On a packing and covering problem.
\newblock {\em European J. Combin.}, 6(1):69--78, 1985.

\bibitem{rogers1947existence}
C.~A. Rogers.
\newblock Existence theorems in the geometry of numbers.
\newblock {\em Ann. of Math.}, 48:994--1002, 1947.

\bibitem{sardari2023new}
N.~T. Sardari and M.~Zargar.
\newblock New upper bounds for spherical codes and packings.
\newblock {\em Mathematische Annalen}, pages 1--51, 2023.

\bibitem{torquato2008}
A.~Scardicchio, F.~H. Stillinger, and S.~Torquato.
\newblock Estimates of the optimal density of sphere packings in high
  dimensions.
\newblock {\em J. Math. Phys.}, 49(4):043301, 15, 2008.

\bibitem{shearer1983note}
J.~B. Shearer.
\newblock A note on the independence number of triangle-free graphs.
\newblock {\em Discrete Math.}, 46:83--87, 1983.

\bibitem{thue1911dichteste}
A.~Thue.
\newblock {\em {\"U}ber die dichteste Zusammenstellung von kongruenten Kreisen
  in einer Ebene}.
\newblock Number~1. J. Dybwad, 1911.

\bibitem{torquato2006new}
S.~Torquato and F.~H. Stillinger.
\newblock New conjectural lower bounds on the optimal density of sphere
  packings.
\newblock {\em Experiment. Math.}, 15(3):307--331, 2006.

\bibitem{vance2011improved}
S.~Vance.
\newblock Improved sphere packing lower bounds from {H}urwitz lattices.
\newblock {\em Adv. Math.}, 227:2144--2156, 2011.

\bibitem{venkatesh2012note}
A.~Venkatesh.
\newblock A note on sphere packings in high dimension.
\newblock {\em Int. Math. Res. Not.}, 2013:1628--1642, 2013.

\bibitem{viazovska2017sphere}
M.~S. Viazovska.
\newblock The sphere packing problem in dimension 8.
\newblock {\em Ann. of Math.}, 185:991--1015, 2017.

\end{thebibliography}
\bibliographystyle{abbrv}

\end{document}